\documentclass[11pt,reqno,oneside]{amsproc}
\title[Navier-Stokes with rotationally invariant force curl]{Regularity of asymptotically axisymmetric Navier-Stokes flows with rotationally invariant curl of the force}

\author[B.~Ingimarson]{Benjamin Ingimarson}
\address{Department of Mathematics, University of Southern California, Los Angeles, CA 90089}
\email{ingimars@usc.edu}

\author[I.~Kukavica]{Igor Kukavica}
\address{Department of Mathematics, University of Southern California, Los Angeles, CA 90089}
\email{kukavica@usc.edu}

\author[Q.~Xu]{Qi Xu}
\address{Department of Mathematics, University of Southern California, Los Angeles, CA 90089}
\email{xuqi@usc.edu}

  \chardef\forshowkeys=0
  \chardef\showllabel=0
  \chardef\refcheck=0
  \chardef\sketches=0
  \chardef\figures=1

\usepackage{enumitem,mathtools,needspace}
\usepackage{datetime}
\usepackage{fancyhdr}
\usepackage{comment}
\usepackage{listings}

\allowdisplaybreaks
\ifnum\forshowkeys=1

  \usepackage[notref,notcite,color]{showkeys}
\fi
\usepackage[margin=1in]{geometry}
\usepackage{amsmath, amsthm, amssymb}
\usepackage{times}
\usepackage{graphicx}
\usepackage[dvipsnames,svgnames,table]{xcolor}
\usepackage{marginnote}
\usepackage[unicode,breaklinks=true,colorlinks=true,linkcolor=blue,urlcolor=blue,citecolor=blue]{hyperref}
\usepackage[most]{tcolorbox}
\usepackage{tikz}

\ifnum\refcheck=1
  \usepackage{refcheck}
\fi

\usepackage{import}
\usepackage{xifthen}
\usepackage{pdfpages}
\usepackage{transparent}

\hypersetup{pdftitle={Regularity of asymptotically axisymmetric Navier-Stokes flows with rotationally invariant curl of the force},pdfauthor={Benjamin Ingimarson, Igor Kukavica, Qi Xu}}
\begin{document}
\def\YY{X}
\def\OO{\mathcal O}
\def\SS{\mathbb S}
\def\CC{\mathbb C}
\def\RR{\mathbb R}
\def\TT{\mathbb T}
\def\ZZ{\mathbb Z}
\def\HH{\mathbb H}
\def\RSZ{\mathcal R}
\def\LL{\mathcal L}
\def\SL{\LL^1}
\def\ZL{\LL^\infty}
\def\GG{\mathcal G}
\def\tt{\langle t\rangle}
\def\erf{\mathrm{Erf}}
\def\mgt#1{\textcolor{magenta}{#1}}
\def\ff{\rho}
\def\gg{G}
\def\sqrtnu{\sqrt{\nu}}
\def\ww{w}
\def\ft#1{#1_\xi}
\def\les{\lesssim}
\def\lec{\lesssim}
\def\ges{\gtrsim}
\def\gec{\gtrsim}
\renewcommand*{\Re}{\ensuremath{\mathrm{{\mathbb R}e\,}}}
\renewcommand*{\Im}{\ensuremath{\mathrm{{\mathbb I}m\,}}}
\ifnum\showllabel=1
 \def\llabel#1{\marginnote{\color{lightgray}\rm\small(#1)}[-0.0cm]\notag}
%%%%%%%%%%%%%%%%%%%\def\llabel#1{\label{#1}}
%  \reversemarginpar
  %\def\llabel#1{\notag}
%  \def\llabel#1{\label{#1}}
\else
% % \def\llabel#1{\nonumber}
 \def\llabel#1{\notag}
\fi
\newcommand{\norm}[1]{\left\|#1\right\|}
\newcommand{\nnorm}[1]{\lVert #1\rVert}
\newcommand{\abs}[1]{\left|#1\right|}
\newcommand{\NORM}[1]{|\!|\!| #1|\!|\!|}

\theoremstyle{plain}
\newtheorem{Theorem}{Theorem}
\newtheorem{Corollary}{Corollary}
\newtheorem{Proposition}{Proposition}
\newtheorem{Lemma}{Lemma}

\theoremstyle{definition}
\newtheorem{Remark}{Remark}
\newtheorem{definition}{Definition}

\def\theequation{\thesection.\arabic{equation}}
\numberwithin{equation}{section}
%%%%%%%%%%%%%%%%%%%%%%%%%%
\definecolor{mygray}{rgb}{.6,.6,.6}
\definecolor{myblue}{rgb}{9, 0, 1}
\definecolor{colorforkeys}{rgb}{1.0,0.0,0.0}
\newlength\mytemplen
\newsavebox\mytempbox
\makeatletter
\newcommand\mybluebox{%
    \@ifnextchar[%]
       {\@mybluebox}%
       {\@mybluebox[0pt]}}
\def\@mybluebox[#1]{%
    \@ifnextchar[%]
       {\@@mybluebox[#1]}%
       {\@@mybluebox[#1][0pt]}}
\def\@@mybluebox[#1][#2]#3{
    \sbox\mytempbox{#3}%
    \mytemplen\ht\mytempbox
    \advance\mytemplen #1\relax
    \ht\mytempbox\mytemplen
    \mytemplen\dp\mytempbox
    \advance\mytemplen #2\relax
    \dp\mytempbox\mytemplen
    \colorbox{myblue}{\hspace{1em}\usebox{\mytempbox}\hspace{1em}}}
\makeatother
%%%%%%%%%%%%%%%%%%%%%%%%%%

%Igor's macros  varmac
\def\pv{\text{p.v.}}
\def\bnew{\color{red}}
\def\enew{\color{black}}
\def\bold{\color{blue}}
\def\eold{\color{black}}
\def\rr{r}
\def\weaks{\text{\,\,\,\,\,\,weakly-* in }}
\def\inn{\text{\,\,\,\,\,\,in }}
\def\cof{\mathop{\rm cof\,}\nolimits}
\def\Dn{\frac{\partial}{\partial N}}
\def\Dnn#1{\frac{\partial #1}{\partial N}}
\def\tdb{\tilde{b}}
\def\tda{b}
\def\qqq{u}
\def\lat{\Delta_2}
\def\biglinem{\vskip0.5truecm\par==========================\par\vskip0.5truecm}
\def\inon#1{\hbox{\ \ \ \ \ \ \ }\hbox{#1}}                %in or on
\def\onon#1{\inon{on~$#1$}}
\def\inin#1{\inon{in~$#1$}}
\def\FF{F}
\def\andand{\text{\indeq and\indeq}}
\def\wherewhere{\text{\indeq where\indeq}}
\def\ww{w(y)}
\def\ll{{\color{red}\ell}}
\def\ee{\mathrm{e}}
\def\startnewsection#1#2{ \section{#1}\label{#2}\setcounter{equation}{0}}   
%\starts a new section\def\dist{\mathop{\rm dist}\nolimits}
\def\nnewpage{ }
\def\sgn{\mathop{\rm sgn\,}\nolimits}    
\def\Tr{\mathop{\rm Tr}\nolimits}    
\def\div{\mathop{\rm div}\nolimits}
\def\curl{\mathop{\rm curl}\nolimits}
\def\dist{\mathop{\rm dist}\nolimits}
\def\id{\mathop{\rm id}\nolimits}
\def\supp{\mathop{\rm supp}\nolimits}
\def\indeq{\quad{}}           
\def\period{.}                       
\def\semicolon{\,;}                  
\def\cmi#1{\text{~{{\coli IK: \underline{#1}}}~}}
\def\coli{\color{colorigor}}
\definecolor{colorigor}{rgb}{.5, 0.2, 0.8}
\def\colr{\color{red}}
\def\colrr{\color{black}}
\def\colb{\color{black}}
\def\coly{\color{lightgray}}
\definecolor{colorgggg}{rgb}{0.1,0.5,0.3}
\definecolor{colorllll}{rgb}{0.0,0.7,0.0}
\definecolor{colorhhhh}{rgb}{0.3,0.75,0.4}
\definecolor{colorpppp}{rgb}{0.7,0.0,0.2}
%\definecolor{coloroooo}{rgb}{0.9,0.4,0}
\definecolor{coloroooo}{rgb}{0.45,0.0,0.0}
\definecolor{colorqqqq}{rgb}{0.1,0.7,0}
\def\colg{\color{colorgggg}}
\def\collg{\color{colorllll}}
\def\cole{\color{coloroooo}}
\def\coleo{\color{colorpppp}}
\def\colu{\color{blue}}
\def\colc{\color{colorhhhh}}
\def\colW{\colb}   %color for weight
\definecolor{coloraaaa}{rgb}{0.6,0.6,0.6}%%%out
\def\colw{\color{coloraaaa}}
\def\comma{ {\rm ,\qquad{}} }            
\def\commaone{ {\rm ,\quad{}} }          
\def\les{\lesssim}
\def\nts#1{{\color{red}\hbox{\bf ~#1~}}} 
\def\ntsf#1{\footnote{\color{colorgggg}\hbox{#1}}} 
\def\blackdot{{\color{red}{\hskip-.0truecm\rule[-1mm]{4mm}{4mm}\hskip.2truecm}}\hskip-.3truecm}
\def\bluedot{{\color{blue}{\hskip-.0truecm\rule[-1mm]{4mm}{4mm}\hskip.2truecm}}\hskip-.3truecm}
\def\purpledot{{\color{colorpppp}{\hskip-.0truecm\rule[-1mm]{4mm}{4mm}\hskip.2truecm}}\hskip-.3truecm}
\def\greendot{{\color{colorgggg}{\hskip-.0truecm\rule[-1mm]{4mm}{4mm}\hskip.2truecm}}\hskip-.3truecm}
\def\cyandot{{\color{cyan}{\hskip-.0truecm\rule[-1mm]{4mm}{4mm}\hskip.2truecm}}\hskip-.3truecm}
\def\reddot{{\color{red}{\hskip-.0truecm\rule[-1mm]{4mm}{4mm}\hskip.2truecm}}\hskip-.3truecm}
\def\tdot{{\color{green}{\hskip-.0truecm\rule[-.5mm]{6mm}{6mm}\hskip.2truecm}}\hskip-.1truecm}
\def\gdot{\greendot}
\def\bdot{\bluedot}
\def\ydot{\cyandot}
\def\rdot{\cyandot}
\def\fractext#1#2{{#1}/{#2}}
\def\ii{\hat\imath}
%%%%%%%
\def\fei#1{\textcolor{blue}{#1}}
\def\vlad#1{\textcolor{cyan}{#1}}
\def\igor#1{\text{{\bf\textcolor{colorqqqq}{IK: #1}}}}
\def\igorf#1{\footnote{\text{{\textcolor{colorqqqq}{#1}}}}}
%varmac
%\def\AA{Y}
\def\Omf{\Omega_{\text f}}
\def\Ome{\Omega_{\text e}}
\def\Omb{\Omega_{\text b}}
\def\Gaf{\Gamma_{\text f}}
\def\Gae{\Gamma_{\text e}}
\def\Gab{\Gamma_{\text b}}
\def\Gac{\Gamma_{\text c}}
\def\Nf{N^{\text f}}
\def\Ne{N^{\text e}}

\newcommand{\p}{\partial}
\renewcommand{\d}{\mathrm{d}}
\newcommand{\dd}{\mathrm{d}}
\newcommand{\UE}{U^{\rm E}}
\newcommand{\PE}{P^{\rm E}}
\newcommand{\KP}{K_{\rm P}}
\newcommand{\uNS}{u^{\rm NS}}
\newcommand{\vNS}{v^{\rm NS}}
\newcommand{\pNS}{p^{\rm NS}}
\newcommand{\omegaNS}{\omega^{\rm NS}}
\newcommand{\uE}{u^{\rm E}}
\newcommand{\vE}{v^{\rm E}}
\newcommand{\pE}{p^{\rm E}}
\newcommand{\omegaE}{\omega^{\rm E}}
\newcommand{\ua}{u_{\rm   a}}
\newcommand{\va}{v_{\rm   a}}
\newcommand{\omegaa}{\omega_{\rm   a}}
\newcommand{\ue}{u_{\rm   e}}
\newcommand{\ve}{v_{\rm   e}}
\newcommand{\omegae}{\omega_{\rm e}}
\newcommand{\omegaeic}{\omega_{{\rm e}0}}
\newcommand{\ueic}{u_{{\rm   e}0}}
\newcommand{\veic}{v_{{\rm   e}0}}
\newcommand{\up}{u^{\rm P}}
\newcommand{\vp}{v^{\rm P}}
\newcommand{\tup}{{\tilde u}^{\rm P}}
\newcommand{\bvp}{{\bar v}^{\rm P}}
\newcommand{\omegap}{\omega^{\rm P}}
\newcommand{\tomegap}{\tilde \omega^{\rm P}}
\newcommand{\eps}{\varepsilon}  %ben's macro
\newcommand{\eqnb}{\begin{equation}}
\newcommand{\eqne}{\end{equation}}
  
\renewcommand{\up}{u^{\rm P}}
\renewcommand{\vp}{v^{\rm P}}
\renewcommand{\omegap}{\Omega^{\rm P}}
\renewcommand{\tomegap}{\omega^{\rm P}}

\begin{abstract}
We answer a question of Constantin, Ignatova, and Vicol concerning
regularity of asymptotically axisymmetric solutions to the
three-dimensional Navier-Stokes equations. For suitable weak solutions
that are smooth before a terminal time and have a force uniformly bounded
in $C_x^2$, we replace spatial analyticity of the force by rotational
invariance of its curl. Under the same anisotropic Type~II bounds on the
angular mean and exact axial symmetry on a possibly collapsing core, the
solution is axisymmetric on the entire cylinder and regular at the putative
singular point. No time regularity or quantitative lower bound on the
core radius is required. The proof combines unique continuation for
linear Stokes systems with a Baire category argument and the
axisymmetric regularity theorem due to Constantin, Ignatova, and Vicol.
\end{abstract}

%\subjclass[2020]{35Q30, 35B60, 35B65}
%\keywords{Navier-Stokes equations, axial symmetry, unique continuation, local regularity}
\maketitle

\section{Introduction}\label{sec01}
In this note, we consider the incompressible 3D Navier-Stokes equations,
\begin{align}\label{EQ00}
\begin{split}
 \partial_t u+(u\cdot\nabla)u-\Delta u+\nabla\pi&=f,\\
 \div u&=0
\end{split}
\end{align}
on the parabolic cylinder $Q=B(1)\times(-1,0)$, where $B(R)\subset\RR^3$ is the ball of radius
$R$ centered at the origin, $u \colon Q \to \mathbb{R}^3$ is the velocity, $\pi \colon Q \to \mathbb{R}$ the pressure, and
$f \colon Q \to \mathbb{R}^3$ is the force.
In particular, we establish a regularity criterion for $u$ under the assumption that $\curl f$ is rotationally invariant about the $x_3$-axis.

Recently, OpenAI announced a proof of finite time singularity formation for the 3D incompressible Navier-Stokes equations with $C^\infty$-smooth forcing, see the manuscript~\cite{OAI}. 
The construction satisfies two specific properties, see (i) and (ii) in Theorem~\ref{T01} below, which were closely examined in a recent paper by Constantin, Ignatova, and Vicol~\cite{CIV1}.
The property (i) consists of anisotropic bounds on the angular mean of
the velocity, while the property (ii) states that $u$ is axisymmetric on a collapsing core~$B(\rho(t))$.
For a verification that the construction in~\cite{OAI} satisfies (i) and (ii), see~\cite[Appendix~A]{CIV1}.
In their paper, they show that if a suitable weak solution $u$ satisfies (i) and (ii), and if the forcing term $f$ is \textit{real-analytic} in space, then the putative singularity $(0,0)$ is, in fact, a regular point.

In Theorem~\ref{T01} below, we consider the same hypotheses as in~\cite[Theorem~1.1]{CIV1} except we replace the assumption of spatially analytic forcing with rotational invariance of its curl about the $x_3$-axis. 
This answers the question posed in~\cite[Remark~2.5]{CIV1} in the affirmative. 
The main ingredient of our proof is unique continuation for solutions to parabolic PDEs.
Indeed, we show that exact axisymmetry on the collapsing core may be extended to the entire parabolic cylinder using a result by Fabre~\cite{F}.

Unique continuation for parabolic PDEs expresses a form of rigidity in which local vanishing propagates beyond the original zero set.
Using Carleman estimates, Fabre~\cite{F} proved that solutions of linear Stokes systems with bounded lower-order coefficients which vanish on an open space-time set must vanish throughout the corresponding connected spatial slices, and applied this result to approximate controllability.
For related unique continuation results for the Stokes system, see Fabre and Lebeau~\cite{FL}.
For quantitative stability estimates for the nonstationary Stokes problem and the lateral Cauchy problem for the linearized Navier-Stokes equations, see Boulakia~\cite{B} and Bellassoued, Imanuvilov, and Yamamoto~\cite{BIY}, respectively.

Next, we provide a brief survey of literature regarding regularity for axisymmetric solutions of Navier-Stokes.
For the unforced equations, Chen, Strain, Yau, and Tsai~\cite{CSYT,CSTY} excluded singularities under Type~I pointwise bounds on the velocity.
Related criteria were obtained by Koch, Nadirashvili, Seregin, and \v Sver\'ak~\cite{KNSS} through Liouville theorems for ancient solutions, and by Seregin and \v Sver\'ak~\cite{SS} under a local Type~I bound on the meridional velocity alone.
Beyond the Type~I setting, Seregin~\cite{Se} used Euler scaling to study certain potential Type~II blowup scenarios.
More directly related to~\cite{CIV1}, Constantin, Ignatova, and Vicol established regularity under Type~II bounds at a vanishing Euler length scale, see~\cite{CIV2}.
For other results on axisymmetric regularity and the developments surrounding the forced blowup construction, we refer to~\cite{CIV1} and the references therein.

The present paper is organized as follows.
In Section~\ref{sec02}, we state our main result, Theorem~\ref{T01}, which states our regularity criterion for $u$ under the assumption that $\curl f$ is rotationally invariant in~$Q$.
We also note in Remark~\ref{R02} that the spatial analyticity assumption in~\cite[Theorem~1.1]{CIV1} may be replaced by a weaker statement. 
Next, we state Theorem~\ref{T02} which is proven in~\cite[Theorem~1.3]{CIV1}. 
In Section~\ref{sec03}, we prove Theorem~\ref{T01} through a unique continuation argument using a result by Fabre, see~\cite[Theorem~1.4]{F}.
We begin the proof by writing a parabolic equation for the difference $u- \mathcal{R}_{\varphi} u$, where $\mathcal{R}_{\varphi} u$ is the rotated velocity. 
The resulting force is irrotational by assumption and, hence, admits an explicit scalar potential~$q$. 
This gives a homogeneous parabolic equation for $u- \mathcal{R}_{\varphi} u $ after absorbing $\nabla q$ into the pressure.
Our main idea is to use unique continuation to expand the exact axisymmetry on the collapsing core region $B(\rho(t))$ to the entire parabolic cylinder. 
However, there is no assumption on the regularity of the time-dependent radius $\rho(t)$ for $t \in (-1,0)$.
Indeed, if $\rho$ were continuous or monotonic, we could employ a compactness argument to obtain an open set $B(r) \times J$, where $J \Subset (-1,0)$ is arbitrary, from which unique continuation yields our result.
Without any regularity of $\rho$, this argument is ineffective.
Instead, we use the Baire category theorem combined with unique continuation~\cite[Theorem~1.4]{F} to show that $u - \mathcal{R}_{\varphi}u$ vanishes on $B(1)\times K$ where $K$ is dense in $(-1,0)$. 
The desired result is then obtained by continuity of~$u$.
It is particularly notable that this theorem and Theorem~1.1 in~\cite{CIV1} hold with no regularity assumptions on the radius~$\rho(t)$.

\section{Preliminaries and main result}\label{sec02}
For $\varphi \in [0,2\pi)$, let $Q_\varphi$ denote rotation about the $z$-axis
and define
\begin{align}\label{EQ104}
 (\mathcal R_\varphi g)(x,t)
 =Q_\varphi g(Q_\varphi^{-1}x,t),
 \qquad
 \mathcal Pg=\frac1{2\pi}\int_0^{2\pi}\mathcal R_\varphi g\,\dd\varphi.
\end{align}
In this paper, a suitable weak solution of~\eqref{EQ00} is a
distributional solution $(u,\pi)$ satisfying
\begin{align}\label{EQ105}
 u\in L^\infty(-1,0;L^2(B(1)))\cap L^2(-1,0;H^1(B(1))),
 \qquad \pi\in L^{3/2}(Q),
\end{align}
and the local energy inequality with source $f\cdot u$, in the sense
of~\cite{CKN}; see also~\cite{L}.
Throughout the paper, we work with solutions
that are smooth on compact subsets of $Q$, so $f$ is smooth there as
well. No regularity at the terminal time $t=0$ is assumed.

\begin{Theorem}\label{T01}
Let $(u,\pi)$ be a suitable weak solution of~\eqref{EQ00} on $Q$, smooth
on compact subsets of $Q$, and let $0 < h < 1/2$. 
Suppose that
\begin{align}\label{EQ100}
 \sup_{-1<t<0}\|f(\cdot,t)\|_{C_x^2(B(1))}<\infty.
\end{align}
Furthermore, assume that the curl of the force is rotationally invariant:
\begin{align}\label{EQ101}
 \mathcal R_\varphi(\curl f)=\curl f
 \qquad\text{on }Q,\qquad \varphi\in[0,2\pi).
\end{align}
Suppose also that the velocity has the following properties.
\begin{itemize}[leftmargin=2em]
\item[(i)] For some time-independent constant $C_{\rm axi}>0$, its angular mean
$v=\mathcal Pu$ satisfies
\begin{align}\label{EQ102}
\begin{split}
 \|\partial_r^a\partial_z^b v_r(\cdot,t)\|_{L^\infty(B(1))}
 &\leq C_{\rm axi}(-t)^{-1/2-a/2-(1/2-h)b},\\
 \|\partial_r^a\partial_z^b v_z(\cdot,t)\|_{L^\infty(B(1))}
 +\|\partial_r^a\partial_z^b v_\theta(\cdot,t)\|_{L^\infty(B(1))}
 &\leq C_{\rm axi}(-t)^{-1/2-h-a/2-(1/2-h)b}
\end{split}
\end{align}
for $-1<t<0$ and all nonnegative integers $a,b$ with $a+b\leq2$.
\item[(ii)] For each $t\in(-1,0)$ there is a radius $\rho(t)\in(0,1)$ such that the solution is exactly axisymmetric on~$B(\rho(t))$.
Equivalently,
\begin{align}\label{EQ106}
 u(\cdot,t)=\mathcal Pu(\cdot,t)
 \qquad\text{on }B(\rho(t)).
\end{align}
No time regularity or quantitative lower bound is imposed on~$\rho$.
\end{itemize}
Then $u=\mathcal Pu$ on $Q$, and $(0,0)$ is a regular point: there exist
$r_\ast,\delta_\ast>0$ such that
\begin{align}\label{EQ103}
 \sup_{B(r_\ast)\times(-\delta_\ast,0)}|u|<\infty.
\end{align}
\end{Theorem}

\begin{Remark}[Invariance under a single rotation]\label{R01}
In~\eqref{EQ101}, it is sufficient to assume invariance under one fixed
rotation through an angle $\varphi_\ast$ with
$\varphi_\ast/(2\pi)\notin\mathbb Q$. Indeed, iterating this rotation
and its inverse gives
\begin{align*}
 \mathcal R_{n\varphi_\ast}(\curl f)=\curl f
 \qquad\text{for every }n\in\mathbb Z.
\end{align*}
The angles $n\varphi_\ast$ modulo $2\pi$ are dense in the circle.
Continuity of $\curl f$ and of the rotation action therefore implies
\eqref{EQ101} for every angle.
\end{Remark}

\begin{Theorem}[{\cite[Theorem~1.3]{CIV1}}]\label{T02}
Let $(u,\pi)$ be an axisymmetric suitable weak solution of~\eqref{EQ00}
on~$Q$, smooth on compact subsets of $Q$, with a force
satisfying~\eqref{EQ100}. Let $0 < h < 1/2$. Suppose that, for some  
$C_{\rm axi}>0$,
\begin{align}\label{EQ107}
\begin{split}
 |\partial_r^a\partial_z^b u_r(r,z,t)|
 &\leq C_{\rm axi}(-t)^{-1/2-a/2-(1/2-h)b},\\
 |\partial_r^a\partial_z^b u_z(r,z,t)|
 +|\partial_r^a\partial_z^b u_\theta(r,z,t)|
 &\leq C_{\rm axi}(-t)^{-1/2-h-a/2-(1/2-h)b}
\end{split}
\end{align}
on $Q$, for all nonnegative integers $a,b$ with $a+b\leq2$.
Then $(0,0)$ is a regular point.
\end{Theorem}

In particular, Theorem~\ref{T02} requires no analyticity of the force.
When applying it, we will average the force and pressure over rotations
so that both are axisymmetric.

\begin{Remark}[Localized analyticity of the force]\label{R02}
The spatial analyticity assumption in~\cite[Theorem~1.1]{CIV1} can
also be localized. 
Indeed, we retain the other hypotheses of that theorem, and
suppose that, for some $R,\tau\in(0,1)$ and a function
$\sigma:(-\tau,0)\to[0,R)$ with $\sigma(t)<\rho(t)$, the force is
real analytic in space, locally uniformly in time, on
\begin{align*}
 \{(x,t):-\tau<t<0,\ \sigma(t)<|x|<R\}.
\end{align*}
Then $u=\mathcal Pu$ on $B(R)\times(-\tau,0)$, and $(0,0)$ is
regular. 
To see this, note that by~\cite[Theorem~2.1]{CIV1}, the difference
$u-\mathcal R_\varphi u$ is real analytic on the connected,
rotation-invariant shell $\{\sigma(t)<|x|<R\}$ at each fixed time.
It vanishes on the nonempty open intersection of this shell with~$B(\rho(t))$.
The identity theorem therefore extends axisymmetry
to the whole shell and, together with the core, to~$B(R)$.
Averaging the pressure and force and applying Theorem~\ref{T02}
on a smaller cylinder by scaling gives regularity.
No analyticity is required in the smaller core $|x|\leq\sigma(t)$;
for example, when $\rho(t)=c\sqrt{-t}$, one may take
$\sigma(t)=\rho(t)/2$ after shortening the time interval.
\end{Remark}

\section{Propagation of symmetry and proof of the main result}\label{sec03}
We first recall the unique continuation theorem in the form needed here.

\begin{Lemma}[{\cite[Theorem~1.4]{F}}]\label{L01}
Let $\Omega\subset\RR^3$ be a bounded connected domain with smooth
boundary and $I$ a bounded open interval.
Let $a^j_{kl},b^j_{kl}\in L^\infty_{\rm loc}(\Omega\times I)$
for $1\leq j,k,l\leq3$, and let
\begin{align*}
 d\in L^2_{\rm loc}(I;H^1_{\rm loc}(\Omega;\RR^3)),
 \qquad p\in L^2_{\rm loc}(\Omega\times I)
\end{align*}
be a distributional solution on $\Omega\times I$ of
\begin{align}\label{EQ301}
\begin{split}
 \partial_t d_j-\Delta d_j
 +\sum_{k,l=1}^3a^j_{kl}\partial_k d_l
 +\sum_{k,l=1}^3\partial_k(b^j_{kl}d_l)
 &=\partial_jp,\qquad j=1,2,3,\\
 \div d&=0
\end{split}
\end{align}
If $d$ vanishes on a nonempty
open set $O\subset\Omega\times I$, then it vanishes on
\begin{align}\label{EQ302}
 C(O)=\{(x,t)\in\Omega\times I:
       (y,t)\in O\text{ for some }y\in\Omega\}.
\end{align}
In particular, if $d=0$ on $B(r)\times I$ with $B(r)\subset\Omega$,
then $d=0$ on~$\Omega\times I$.
\end{Lemma}

The next proposition contains the main argument using unique continuation.

\begin{Proposition}\label{P01}
Let $(u,\pi)$ be a solution of~\eqref{EQ00} that is smooth on compact subsets
of~$Q$. If~\eqref{EQ101} and~\eqref{EQ106} hold, then
$u=\mathcal Pu$ on~$Q$.
\end{Proposition}

\begin{proof}
Fix $\varphi\in[0,2\pi)$, and set
\begin{align}\label{EQ303}
 U=\mathcal R_\varphi u,\qquad
 \pi_\varphi(x,t)=\pi(Q_\varphi^{-1}x,t),\qquad d=u-U.
\end{align}
Rotation covariance of~\eqref{EQ00} gives
\begin{align*}
 \partial_t U+(U\cdot\nabla)U-\Delta U+\nabla\pi_\varphi
 =\mathcal R_\varphi f,\qquad \div U=0.
\end{align*}
Subtracting this equation from~\eqref{EQ00} and using
\begin{align*}
 (u\cdot\nabla)u-(U\cdot\nabla)U
 =(u\cdot\nabla)d+(d\cdot\nabla)U,
\end{align*}
we obtain
\begin{align}\label{EQ304}
 \partial_t d-\Delta d+(u\cdot\nabla)d+(d\cdot\nabla)U
 +\nabla(\pi-\pi_\varphi)=g,
 \wherewhere
  g=f-\mathcal R_\varphi f.
\end{align}
Since proper rotations commute with $\curl$,~\eqref{EQ101} implies~$\curl g=0$. 
Next, we define
\begin{align}\label{EQ305}
 \chi(x,t)=\int_0^1 g(sx,t)\cdot x\,\dd s.
\end{align}
The identity $\partial_jg_i=\partial_ig_j$ yields
\begin{equation*}
 \partial_j\chi(x,t)
 =\int_0^1\biggl(g_j(sx,t)
       +s\sum_{i=1}^3x_i\partial_jg_i(sx,t)\biggr)\dd s\\
 =\int_0^1\frac{\dd}{\dd s}\bigl(sg_j(sx,t)\bigr)\,\dd s
 =g_j(x,t).
\end{equation*}
Thus, with $q=\pi-\pi_\varphi-\chi$, the equation~\eqref{EQ304} becomes
\begin{align}\label{EQ306}
 \partial_t d-\Delta d+(u\cdot\nabla)d+(d\cdot\nabla)U
 =-\nabla q,\qquad \div d=0.
\end{align}
Since $\div d=0$, we have
$(d\cdot\nabla)U_j=\sum_{k=1}^3\partial_k(U_jd_k)$.
Consequently,~\eqref{EQ306} has the form~\eqref{EQ301} with
\begin{align}\label{EQ307}
 a^j_{kl}=u_k\delta_{jl},\qquad
 b^j_{kl}=U_j\delta_{kl},\qquad p=-q.
\end{align}
Since $(u,\pi)$ and hence $f$ and $\chi$ are smooth in $Q$, all local
regularity assumptions of Lemma~\ref{L01}, including
$p=-q\in L^2_{\rm loc}(Q)$, hold.

We next find open space-time sets on which $d$ vanishes. By
\eqref{EQ106}, for every $t\in(-1,0)$ we have
$d(\cdot,t)=0$ on~$B(\rho(t))$. Let $I\subset(-1,0)$ be an arbitrary
nonempty open interval, and choose a nondegenerate compact interval~$J\subset I$. For $m\geq2$, set
\begin{align*}
 E_m=\{t\in J:d(x,t)=0\text{ for every }x\in B(1/m)\}.
\end{align*}
Each $E_m$ is closed in $J$ by continuity, and
$J=\bigcup_{m\geq2}E_m$ because $\rho(t)>0$. The Baire category theorem
therefore gives an $m$ and a nonempty open interval $I'\subset I$ such
that $I'\subset E_m$. In particular,
\begin{align}\label{EQ308}
 d=0\qquad\text{on }B(1/m)\times I'.
\end{align}
For $1/m<R<1$, apply Lemma~\ref{L01} on
$\Omega\times I'=B(R)\times I'$ with
$O=B(1/m)\times I'$. Its horizontal component is precisely
$C(O)=B(R)\times I'$. Hence, $d=0$ on $B(R)\times I'$, and letting
$R\uparrow1$ gives $d=0$ on~$B(1)\times I'$.
It follows that
\begin{align*}
 D=\{t\in(-1,0):d(x,t)=0\text{ for every }x\in B(1)\}
\end{align*}
meets every nonempty open time interval and is therefore dense in~$(-1,0)$.
It is also closed in $(-1,0)$ by continuity. Thus $D=(-1,0)$,
and $u=\mathcal R_\varphi u$ on~$Q$. Since $\varphi$ was arbitrary,
averaging over rotations gives $u=\mathcal Pu$ on~$Q$.
\end{proof}

\begin{proof}[Proof of Theorem~\ref{T01}]
By Proposition~\ref{P01}, we have $u=v$ on $Q$, so that
$u$ satisfies~\eqref{EQ107}. Define the averaged pressure
and force by
\begin{align}\label{EQ309}
 \bar\pi(x,t)=\frac1{2\pi}\int_0^{2\pi}
          \pi(Q_\varphi^{-1}x,t)\,\dd\varphi,
 \qquad \bar f=\mathcal Pf.
\end{align}
Since $u$ is axisymmetric, averaging~\eqref{EQ00} leaves its velocity terms unchanged, and we obtain
\begin{align*}
 \partial_tu+(u\cdot\nabla)u-\Delta u+\nabla\bar\pi=\bar f,
 \qquad \div u=0.
\end{align*}
Note that both $\bar\pi$ and $\bar f$ are axisymmetric; moreover,
\begin{equation*}
 \|\bar\pi\|_{L^{3/2}(Q)}\leq\|\pi\|_{L^{3/2}(Q)}
 ,
\end{equation*}
and
\begin{equation*}
 \sup_{-1<t<0}\|\bar f(\cdot,t)\|_{C_x^2(B(1))}
 \leq C\sup_{-1<t<0}\|f(\cdot,t)\|_{C_x^2(B(1))}<\infty
 .
\end{equation*}
The averaged equation is smooth on compact subsets of $Q$, so its
local energy equality holds. Together with~\eqref{EQ105}, this
shows that $(u,\bar\pi)$ is a suitable weak solution with force~$\bar f$.
All hypotheses of Theorem~\ref{T02} are now satisfied, and
\eqref{EQ103} follows.
\end{proof}

\par\medskip\noindent\textbf{Acknowledgments.} The authors were supported in part by the NSF grant DMS-2205493
and the Simons grant SFI-MPS-TSM-00025796.
We used ChatGPT for help with writing and some arguments.


\begin{thebibliography}{CSTY1}
\small

\bibitem[BIY]{BIY}
M.~Bellassoued, O.~Imanuvilov, and M.~Yamamoto,
\emph{Carleman estimate for the Navier-Stokes equations and an application
to a lateral Cauchy problem},
Inverse Problems \textbf{32} (2016), no.~2, 025001, 23~pp.

\bibitem[B]{B}
M.~Boulakia,
\emph{Quantification of the unique continuation property for the
nonstationary Stokes problem},
Math. Control Relat. Fields \textbf{6} (2016), no.~1, 27--52.

\bibitem[CKN]{CKN}
L.~Caffarelli, R.~Kohn, and L.~Nirenberg,
\emph{Partial regularity of suitable weak solutions of the Navier-Stokes
equations},
Comm. Pure Appl. Math. \textbf{35} (1982), no.~6, 771--831.

\bibitem[CSYT]{CSYT}
C.-C.~Chen, R.~M.~Strain, H.-T.~Yau, and T.-P.~Tsai,
\emph{Lower bound on the blow-up rate of the axisymmetric Navier-Stokes
equations},
Int. Math. Res. Not. IMRN (2008), no.~9, Art.~ID~rnn016, 31~pp.

\bibitem[CSTY]{CSTY}
C.-C.~Chen, R.~M.~Strain, T.-P.~Tsai, and H.-T.~Yau,
\emph{Lower bounds on the blow-up rate of the axisymmetric Navier-Stokes
equations~II},
Comm. Partial Differential Equations \textbf{34} (2009), no.~3, 203--232.

\bibitem[CIV1]{CIV1}
P.~Constantin, M.~Ignatova, and V.~Vicol,
\emph{Regularity of asymptotically axisymmetric solutions to the 3D
Navier-Stokes equations with analytic forcing},
\href{https://arxiv.org/abs/2609.20803v2}{arXiv:2609.20803v2} (2026).

\bibitem[CIV2]{CIV2}
P.~Constantin, M.~Ignatova, and V.~Vicol,
\emph{Regularity for axisymmetric Navier-Stokes with an Euler length},
\href{https://arxiv.org/abs/2609.20762v2}{arXiv:2609.20762v2} (2026).

\bibitem[ESS]{ESS}
L.~Escauriaza, G.~A.~Seregin, and V.~\v Sver\'ak,
\emph{$L_{3,\infty}$-solutions of the Navier-Stokes equations and
backward uniqueness},
Russian Math. Surveys \textbf{58} (2003), no.~2, 211--250.

\bibitem[F]{F}
C.~Fabre,
\emph{Uniqueness results for Stokes equations and their consequences
in linear and nonlinear control problems},
ESAIM Control Optim. Calc. Var. \textbf{1} (1996), 267--302,
doi:\href{https://doi.org/10.1051/cocv:1996109}{10.1051/cocv:1996109}.

\bibitem[FL]{FL}
C.~Fabre and G.~Lebeau,
\emph{Prolongement unique des solutions de l'\'equation de Stokes},
Comm. Partial Differential Equations \textbf{21} (1996), nos.~3--4,
573--596.

\bibitem[KNSS]{KNSS}
G.~Koch, N.~Nadirashvili, G.~A.~Seregin, and V.~\v Sver\'ak,
\emph{Liouville theorems for the Navier-Stokes equations and applications},
Acta Math. \textbf{203} (2009), no.~1, 83--105.

\bibitem[L]{L}
F.-H.~Lin,
\emph{A new proof of the Caffarelli-Kohn-Nirenberg theorem},
Comm. Pure Appl. Math. \textbf{51} (1998), no.~3, 241--257.

\bibitem[OAI]{OAI}
OpenAI,
\emph{Finite time blowup for Navier-Stokes},
manuscript, September~2026, \\
\url{https://openai.com/index/navier-stokes-solution/}.

\bibitem[Se]{Se}
G.~Seregin,
\emph{Remarks on Type~II blowups of solutions to the Navier-Stokes equations},
Commun. Pure Appl. Anal. \textbf{23} (2024), no.~10, 1389--1406.

\bibitem[SS]{SS}
G.~Seregin and V.~\v Sver\'ak,
\emph{On Type~I singularities of the local axi-symmetric solutions of
the Navier-Stokes equations},
Comm. Partial Differential Equations \textbf{34} (2009), no.~2, 171--201.

\end{thebibliography}
\end{document}